\documentclass[10pt]{amsart}
\usepackage{amsmath,amsthm, epsfig, amscd}
\usepackage{psfrag}
\usepackage[psamsfonts]{amssymb}
\usepackage[all]{xy}
\usepackage{mathrsfs}

\numberwithin{equation}{section}

\newtheorem{Thm}{Theorem}[section]
\newtheorem{Prop}[Thm]{Proposition}
\newtheorem{Lem}[Thm]{Lemma}
\newtheorem{Cor}[Thm]{Corollary}
\newtheorem{Conj}[Thm]{Conjecture}

\theoremstyle{definition}

\newtheorem{Def}[Thm]{Definition}

\newtheorem{Exa}[Thm]{Example}

\newcommand{\mysection}[2]{%
\vspace{2mm}\section{\bf #1}\label{#2}
}

\newcommand{\fig}[1]
        {\raisebox{-0.5\height}
                 {\includegraphics{#1}}
        }

\def\Z{{\mathbb Z}}
\def\R{{\mathbb R}}
\def\Q{{\mathbb Q}}

\def\calA{\mathscr{A}}

\def\calE{\mathscr{E}}

\def\calG{\mathscr{G}}

\def\calK{\mathscr{K}}

\def\calM{\mathscr{M}}

\def\calS{\mathscr{S}}

\def\deg{\mathrm{deg}}

\def\Aut{\mathrm{Aut}\,}

\newcommand{\mapright}[1]{
	\smash{\mathop{
		\hbox to 1cm{\rightarrowfill}}\limits^{#1}}}

\newcommand{\mapleft}[1]{
	\smash{\mathop{
		\hbox to 1cm{\leftarrowfill}}\limits^{#1}}}

\def\Tr{\mathrm{Tr}}

\def\Conf{\mathrm{Conf}}

\def\bcalM{\overline{\calM}}

\def\ve{\varepsilon}
\def\bConf{\overline{\Conf}}

\def\acalM{\calM^{\mathrm{Z}}}
\def\bacalM{\bcalM^{\mathrm{Z}}}

\def\wM{\widetilde{M}}
\def\wf{\widetilde{f}}

\def\twedge{\textstyle\bigwedge}

\def\irr{\mathrm{irr}}

\def\null{\mathrm{NH}}
\def\wxi{\widetilde{\xi}}

\begin{document}

\title[Finite type invariants of nullhomologous knots]{Finite type invariants of nullhomologous knots in 3-manifolds fibered over $S^1$ by counting graphs}
\author{Tadayuki Watanabe}
\address{Department of Mathematics, Shimane University,
1060 Nishikawatsu-cho, Matsue-shi, Shimane 690-8504, Japan}
\email{tadayuki@riko.shimane-u.ac.jp}
\date{\today}

{\noindent\footnotesize {\rm Preprint} (2015)}\par\vspace{17mm}
\maketitle
\vspace{-6mm}
\setcounter{tocdepth}{2}
\begin{abstract}
We study finite type invariants of nullhomologous knots in a closed 3-manifold $M$ defined in terms of certain descending filtration $\{\calK_n(M)\}_{n\geq 0}$ of the vector space $\calK(M)$ spanned by isotopy classes of nullhomologous knots in $M$. The filtration $\{\calK_n(M)\}_{n \geq 0}$ is defined by surgeries on special kinds of claspers in $M$ having one special leaf. More precisely, when $M$ is fibered over $S^1$ and $H_1(M)=\Z$, we study how far the natural surgery map from the space of $\Q[t^{\pm 1}]$-colored Jacobi diagrams on $S^1$ of degree $n$ to the graded quotient $\calK_n(M)/\calK_{n+1}(M)$ can be injective for $n\leq 2$. To do this, we construct a finite type invariant of nullhomologous knots in $M$ up to degree 2 that is an analogue of the invariant given in our previous paper arXiv:1503.08735, which is based on Lescop's construction of $\Z$-equivariant perturbative invariant of 3-manifolds.
\end{abstract}
\par\vspace{3mm}

\def\baselinestretch{1.07}\small\normalsize


\mysection{Introduction}{s:***}

There is a natural descending filtration $\calK_0(S^3)\supset\calK_1(S^3)\supset \calK_2(S^3)\supset\cdots$ on the vector space (or $\Z$-module) $\calK(S^3)$ spanned by isotopy classes of knots in $S^3$, called the Vassiliev filtration (after Vassiliev's \cite{Va1}), whose $n$-th term is spanned by alternating sums of possible resolutions of singular knots with $n$ double points. In terms of the Vassiliev filtration, finite type (or Vassiliev) invariant of knots of degree $n$ is defined as linear maps from $\calK(S^3)/\calK_{n+1}(S^3)$ (\cite{BL, BN2}). It is known that the natural ``geometric realization'' map gives an isomorphism from the vector space of certain trivalent graphs called the Jacobi diagrams (\cite{BN1, BN2}, see also \cite{CDM}) to the graded quotients $\calK_n(S^3)/\calK_{n+1}(S^3)$. This very striking result has been proved by Kontsevich in \cite{Ko2} with the help of his diagram-valued universal finite type invariant of knots in $S^3$, so called the Kontsevich integral of knots. There is another construction of finite type invariants of knots in $S^3$ coming from Chern--Simons perturbation theory, which are given by integrations on configuration spaces (\cite{BN1, GMM, Koh, Ko1, BT}). It has been proved by Kontsevich \cite{Ko1} (degree 2) and Altschuler--Freidel \cite{AF} (all degrees) that the configuration space integral invariant give another universal finite type invariant of knots in $S^3$. 

In this paper, we study finite type invariants of nullhomologous knots in an oriented closed 3-manifold $M$. As an analogoue of finite type invariants of knots in $S^3$ defined by null-claspers (Garoufalidis--Rozansky \cite{GR}), we introduce a descending filtration $\calK_0(M)\supset \calK_1(M)\supset \calK_2(M)\supset \cdots $ of the space of isotopy classes of nullhomologous knots in $M$ by using surgeries on certain claspers with nullhomologous leaves. In the case $H_1(M)=\Z$ and $M$ is fibered over $S^1$, we show that the natural surgery map from the space of $\Q[t^{\pm 1}]$-colored Jacobi diagrams on $S^1$ of degree $n$ to the graded quotient $\calK_n(M)/\calK_{n+1}(M)$ is injective for $n=1$. For $n=2$, we also have similar but weaker statement. The main idea is to construct a diagram-valued perturbative invariants of nullhomologous knots in a fibered 3-manifold $M$ by a method similar to \cite{Les2, Les3, Wa3}. 

We remark that finite type invariants of knots in general 3-manifolds have been developed in \cite{Ka, KL, Va2} by considering singular knots with double points. Although the definition of finite type invariants of \cite{Ka, KL, Va2} is different from that studied in this paper, it can be considered as a ``noncommutative'' refinement of ours. We also remark that in \cite{Ha}, Habiro studied finite type invariants of links in general 3-manifolds defined by using surgeries on graph claspers, which is different from that studied in this paper too. Although Habiro's definition of finite type invariant is natural from the point of view of clasper theory, we modify Habiro's definition for our purpose. Other relevant works can be found in \cite{Sch1, Sch2, Lie}.

We define the perturbative invariant as the trace of the generating function of counts of certain graphs in $M$, which we call {\it Z-graphs}. The definition of our invariant is based on Lescop's works on equivariant perturbative invariants of knots and 3-manifolds \cite{Les2, Les3, Les4} and on the explicit propagator of ``Z-paths'' in $M$ given in \cite{Wa2} using parametrized Morse theory. Our construction can be considered as an analogue of that of the configuration space integral invariant. Our knot invariant differs from Lescop's one in the presence of the Wilson loop in Jacobi diagrams, which is a simple distinguished cycle. 

Throughout this paper, manifolds and maps between them are smooth unless otherwise indicated. We use the outward-normal-first convention to orient boundaries of manifolds. Homology groups are assumed to be with integer coefficients unless otherwise specified. A knot will denote both an embedding $S^1\to M$ and its image in $M$. We consider a graph as a topological space by identifying it with its geometric realization. 

\mysection{Finite type invariants of nullhomologous knots in $M$}{s:***}

We study finite type invariants of nullhomologous knots in a closed 3-manifold defined by using clasper surgeries. The theory of clasper surgery has been developed independently by Goussarov and Habiro in \cite{Gu, Ha}. We review some fundamental properties of claspers and prescribe the type of finite type invariants studied in the present paper. The main result of the present paper is stated in terms of clasper surgeries. 

\subsection{Null-claspers, filtration $\calK_n(M)$ and finite type invariant}

We shall recall definition of clasper surgery from \cite{Ha}. Let $M$ be a closed connected oriented Riemannian 3-manifold. For simplicity, we assume that $H_1(M)$ is free abelian. A {\it tree clasper for a knot $K$ in $M$} is a compact connected surface immersed in $M$ consisting of {\it bands, nodes, leaves} and {\it disk-leaves} (as in Figure~\ref{fig:tree-clasper}). A band is an embedded 2-dimensional 1-handle, a node is an embedded 2-dimensional 0-handle on which three bands are attached. A leaf (resp. disk-leaf) is an embedded annulus (resp. an embedded 0-handle) on which a band is attached. The union of bands, nodes and leaves is embedded in $M\setminus K$, whereas a disk-leaf may intersect $K$ or leaves or bands transversally in its interior. 

A tree clasper without nodes is called an {\it $I$-clasper}. On an $I$-clasper, surgery is defined as follows. In this paper, we only consider $I$-clasper with at least one disk-leaves. Then surgery on such an $I$-clasper $C$ is defined as the replacement as in Figure~\ref{fig:surgery-I}. One may also consider surgery on a tree clasper by replacing nodes and disk-leaves with collections of leaves as in Figure~\ref{fig:tree-clasper}, which can be realized by iterated applications of the moves in Figure~\ref{fig:surgery-I}. According to \cite{Ha}, the result of the surgery is determined uniquely up to isotopy.
\begin{figure}
\fig{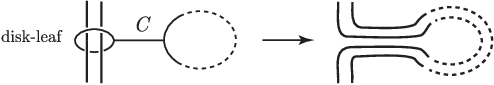}
\caption{$I$-clasper and surgery on it (we use the blackboard-framing convention to represent ribbon structure of $I$-clasper)}\label{fig:surgery-I}
\end{figure}

We say that a tree clasper for a nullhomologous knot $K$ in $M$ is {\it $M$-null} if its leaves consist of disk-leaves except at most one leaf that is nullhomologous in $M$. See Figure~\ref{fig:tree-clasper} for an example. A {\it strict tree clasper} is a tree clasper with only disk-leaves that intersect only with the knot $K$. The {\it degree} of a tree clasper $T$ is the number of nodes in $T$ plus 1. Let $K^T$ denote the knot in $M$ obtained from $K$ by surgery along the set of $I$-claspers associated to $T$. 
\begin{figure}
\fig{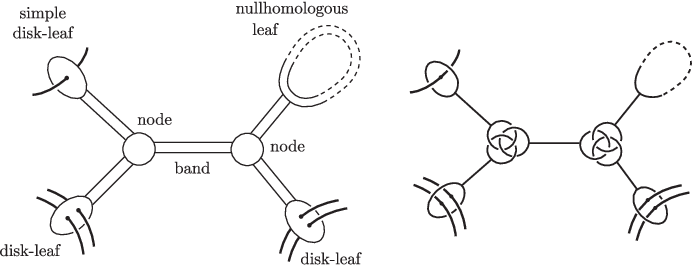}
\caption{$M$-null tree clasper and the associated set of $I$-claspers}\label{fig:tree-clasper}
\end{figure}

\begin{Lem}\label{lem:unknotting}
If $T$ is an $M$-null tree clasper for a nullhomologous knot $K$ in $M$, then $K^T$ is again a nullhomologous knot in $M$. Two nullhomologous knots in $M$ are related by surgeries on finitely many strict $I$-claspers and at most one $M$-null $I$-clasper.
\end{Lem}
\begin{proof}
By Habiro's move 9 in \cite[Proposition~2.7]{Ha}, surgery on an $M$-null tree clasper can be replaced with a sequence of surgeries on $M$-null $I$-claspers, namely, $I$-claspers with only disk-leaves or that with one disk-leaf and one nullhomologous leaf in $M$. 
\[ \fig{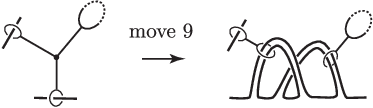} \]
It is obvious that surgeries on such $I$-claspers do not change the homology class of a knot in $M$.

For the second assertion, let $K_0,K_1$ be two nullhomologous knots in $M$. We may assume without loss of generality that $K_0$ and $K_1$ are mutually disjoint. Take base points $q_0,q_1$ on $K_0,K_1$ respectively and a path $\gamma_{01}$ in $M$ going from $q_0$ to $q_1$ such that $\mathrm{Im}\,\gamma_{01}\cap(K_0\cup K_1)=\{q_0,q_1\}$. Let $K_{01}$ be a knot in $M$ given by a connected sum $(-K_0)\#_{\gamma_{01}}K_1$ taken along $\gamma_{01}$. Let $C$ be an $M$-null $I$-clasper whose disk-leaf $\delta$ intersects $K_0$ at $q_0$ and the other leaf is a parallel of $K_{01}$ disjoint from $K_0\cup K_1\cup \delta$. Then $K_0^C$ is homotopic to $K_1$. Hence $K_1$ is obtained from $K_0$ by surgery on $C$ and several crossing changes.  
\end{proof}

Lemma~\ref{lem:unknotting} motivates the definition of finite type invariants given below. Let $\calK(M)$ be the vector space over $\Q$ spanned by isotopy classes of all nullhomologous knots in $M$. Let $\calK_{n,k}(M)$ ($1\leq k\leq n$) be the subspace of $\calK(M)$ spanned by 
\[ [K;G]=\sum_{I\subset \{1,2,\ldots,k\}}(-1)^{k-|I|} K^{G_I},\quad\mbox{where} \]
\begin{itemize}
\item $K$ is a nullhomologous knot in $M$,
\item $G=\{G_1,G_2,\ldots,G_k\}$ is a disjoint collection of tree claspers with $\sum_{i=1}^k \deg\,G_i=n$ that consists of strict tree claspers and at most one $M$-null tree clasper. 
\item $G_I=\bigcup_{i\in I}G_i$.
\end{itemize}
The alternating sum $[K;G]$ as above is called an {\it $M$-null forest scheme of degree $n$, size $k$}. When all $G_i$ are strict, then $[M;G]$ is called a {\it strict forest scheme}. We put
\[  \calK_0(M)=\calK(M),\qquad \calK_n(M)=\sum_{k=1}^n \calK_{n,k}(M)\quad\mbox{(for $n\geq 1$)}.\] 
The following lemma follows as a corollary of the results in \cite[page 48]{Ha}.
\begin{Lem}\label{lem:prop-scheme}
\begin{enumerate}
\item If $1\leq k\leq k'\leq n$, then $\calK_{n,k}(M)\subset \calK_{n,k'}(M)$. In particular, $\calK_n(M)=\calK_{n,n}(M)$.
\item $[K;G_1,G_2,\ldots,G_k]=[K^{G_1};G_2,\ldots,G_k]-[K;G_2,\ldots,G_k]$.
\item $[K;S\cup T,G_2,\ldots,G_k]=[K;S,G_2,\ldots,G_k]+[K^S;T,G_2,\ldots,G_k]$, where $S$ and $T$ are tree claspers that are disjoint.
\end{enumerate}
\end{Lem}

\begin{Def}\label{def:FTI}
Let $V$ be a vector space over $\Q$ and let $n\geq 0$. We say that a linear map 
\[ f:\calK(M)\to V\]
is a {\it finite type invariant of $M$-null type $n$} if $f(\calK_{n+1}(M))=0$. 
\end{Def}

A fundamental problem in the thoery of finite type invariant is to determine the structure of the quotient space $\calK_n(M)/\calK_{n+1}(M)$. The restriction of tree claspers given in the definition of $\calK_{n,k}(M)$ may not look natural. However, this definition is nice to relate $\calK_n(M)/\calK_{n+1}(M)$ with the space of $\Lambda$-colored Jacobi diagrams defined below.
\subsection{$\Lambda$-colored Jacobi diagrams}

A {\it Jacobi diagram on $S^1$} is a connected trivalent graph with oriented edges and with a distinguished simple oriented cycle, called the {\it Wilson loop} (\cite{BN1,BN2}, see also \cite[Ch.5]{CDM}). See Figure~\ref{fig:chord-diags} for examples of Jacobi diagrams with few vertices. A {\it labeled Jacobi diagram} is a Jacobi diagram $\Gamma$ equipped with bijections $\alpha:\{1,2,\ldots,2n\}\to V(\Gamma)$ and $\beta:\{1,2,\ldots,3n\}\to E(\Gamma)$, where $V(\Gamma)$ (resp. $E(\Gamma)$) is the set of vertices (resp. edges) of $\Gamma$ (including the edges in the Wilson loop). Let $E^W(\Gamma)$ be the subset of $E(\Gamma)$ consisting of edges in the Wilson loop and let $E^{nW}(\Gamma)=E(\Gamma)\setminus E^W(\Gamma)$. Let $V^W(\Gamma)$ be the subset of $V(\Gamma)$ consisting of vertices on the Wilson loop and let $V^{nW}(\Gamma)=V(\Gamma)\setminus V^W(\Gamma)$. Let $P(\Gamma)$ be the set of components in the subgraph of $\Gamma$ formed by $E^{nW}(\Gamma)$. A Jacobi diagram $\Gamma$ on $S^1$ with $V^{nW}(\Gamma)=\emptyset$ is called a {\it chord diagram}. A {\it vertex-orientation} of a trivalent vertex $v\in V^{nW}(\Gamma)$ is a cyclic ordering of the edges incident to $v$. A Jacobi diagram all of whose trivalent vertices are equipped with vertex-orientations is said to be vertex-oriented. 

Let $R$ be a commutative ring with 1. For a Jacobi diagram $\Gamma$ on $S^1$, an {\it $R$-coloring} of $\Gamma$ is an assignment of an element of $R$ to every edge of $\Gamma$. An $R$-coloring is represented by a map $\phi:E(\Gamma)\to R$. The {\it degree} of a Jacobi diagram is defined as half the number of vertices. A vertex of $\Gamma$ that is on the Wilson loop is called a {\it univalent vertex} and otherwise a {\it trivalent vertex}. The following definition is an analogue of the graphs considered in \cite{GR}. We write $H=H_1(M)$ and choose a basis $\{t_{(1)},t_{(2)},\ldots,t_{(k)}\}$ of $H$.

\begin{Def}
Let $\Lambda_M$ be the group ring $\Q[H]$. Let $\calA_n(S^1;\Lambda_M)$ be the vector space over $\Q$ spanned by pairs $(\Gamma,\phi)$, where $\Gamma$ is a Jacobi diagram of degree $n$ with vertex-orientation and $\phi$ is a $\Lambda_M$-coloring (resp. ) of $\Gamma$, quotiented by the relations AS, IHX, STU, FI, Orientation reversal, Linearity, Holonomy (Figure~\ref{fig:relations} and \ref{fig:relations2}) and automorphisms of oriented graphs.
\end{Def}
There is a canonical way to determine an equivalence class of vertex-orientation modulo the AS relation from labelings $\alpha,\beta$ on a Jacobi diagram (see e.g., \cite{CV}). 

\begin{Def}
Let $\widehat{\Lambda}_M$ be the total ring of fractions $Q(\Lambda_M)=\Q(H)$. Let $\calA_n(S^1;\widehat{\Lambda}_M)$ be the vector space over $\Q$ spanned by pairs $(\Gamma,\phi)$, where $\Gamma$ is a Jacobi diagram of degree $n$ with vertex-orientation and $\phi$ is a coloring $E(\Gamma)\to \widehat{\Lambda}_M$ such that $\phi(E^W(\Gamma))\subset\Lambda_M$, quotiented by the relations AS, IHX, STU, FI, Orientation reversal, Linearity, Holonomy (Figure~\ref{fig:relations} and \ref{fig:relations2}) and automorphisms of oriented graphs. 
\end{Def}

\begin{figure}
\fig{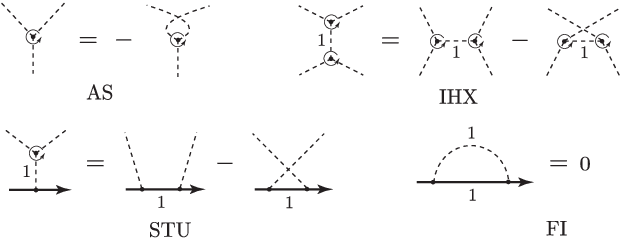}
\caption{The relations AS, IHX, STU and FI.}\label{fig:relations}
\end{figure}

\begin{figure}
\fig{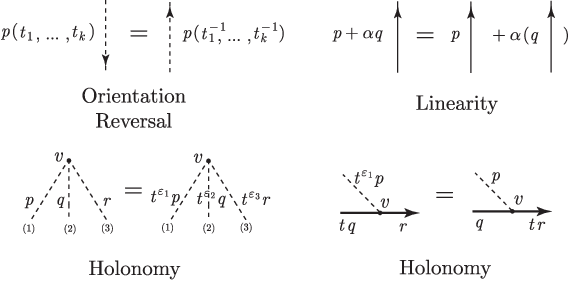}
\caption{The relations Orientation reversal, Linearity and Holonomy. $p,q,r\in \Lambda_M$ (or $p,q,r\in \widehat{\Lambda}_M$), $\alpha\in\Q$, $h=t_{(i)}$ for $i\in\{1,\ldots,r\}$. The exponent $\ve_i$ is $1$ if the $i$-th edge is oriented toward $v$ and otherwise $-1$. $\bar{p}$ is defined by $\bar{p}(t_{(1)},\ldots,t_{(r)})=p(t_{(1)}^{-1},\ldots,t_{(r)}^{-1})$. The edge in the Linearity relation is either of $E^W(\Gamma)$ or of $E^{nW}(\Gamma)$. }\label{fig:relations2}
\end{figure}

We denote a pair $(\Gamma,\phi)$ by $\Gamma(\phi)$ or by $\Gamma(\phi(e_1),\phi(e_2),\ldots,\phi(e_\ell))$. We say that a $\Lambda_M$-colored graph $\Gamma(\phi)$ is a {\it monomial} if for each edge $e$ of $\Gamma$, $\phi(e)$ is an element of $H$. In this case, we may consider $\phi$ as a map $E(\Gamma)\to H$. There is a bijective correspondence between the equivalence class of a labeled monomial Jacobi diagram $\Gamma(\phi)$ modulo the Holonomy relation and the homotopy class of a continuous map $c:\Gamma\to K(H,1)$, or the cohomology class $[c]\in H^1(\Gamma;H)=[\Gamma,K(H,1)]$. We say that $\Gamma(\phi)$ is nullhomotopic if $c$ is nullhomotopic.

\begin{Def}
Let $\calA_n^\null(S^1;\Lambda_M)$ be the subspace of $\calA_n(S^1;\Lambda_M)$ spanned by the set $\calG_n^\null(S^1;\Lambda_M)$ of all monomial Jacobi diagrams such that\footnote{We consider the group structure of $H_1(M)$ as multiplication.} $\prod_{e\in E^W(\Gamma)}\phi(e)=1$. 
\end{Def}
Note that the natural map $\calA_n^\null(S^1;\Lambda_M)\to \calA_n(S^1;\widehat{\Lambda}_M)$ may not be injective, as pointed out in \cite[Remark~2.1]{Les3}.

\subsection{Surgery map $\psi_n$}

Let $\iota:M\to K(H,1)$ be a map that represents the class in $H^1(M;H)=[M,K(H,1)]$ corresponding to the identity in $\mathrm{Hom}(H_1(M),H_1(M))= H^1(M;H_1(M))$. Consider a monomial $\Lambda_M$-colored Jacobi diagram $\Gamma\in \calG_n^\null(S^1;\Lambda_M)$ and a piecewise smooth embedding $\rho:\Gamma\to M$ that preserves its vertex-orientation, namely, the wedge of outward tangent vectors of the three edges at a vertex gives the orientation of $M$, and such that the homotopy class of the composition $\Gamma \stackrel{\rho}{\to} M \stackrel{\iota}{\to} K(H,1)$ represents the $\Lambda_M$-coloring of $\Gamma$. We assign a collection of tree claspers $G=\{G_1,\ldots,G_k\}$ with only disk-leaves to $\rho$ as follows. We replace vertices in the embedded graph $\rho(\Gamma)$ with leaves or nodes of claspers as follows:
\[ \begin{split}
&\fig{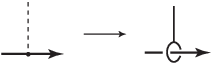}\\
&\fig{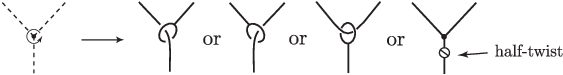}
\end{split} \]
so that each component in $P(\Gamma)$ is mapped to a (connected) tree clasper. Here is an example.
\[ \fig{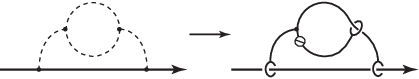} \]
One may see that the alternating sum $[K;G]$ represents an element of $\calK_{n,n}(M)$ by applying Habiro's move 9 of \cite[Proposition~2.7]{Ha} several times. 
\begin{Prop}\label{prop:psi-surjective}
The assignment $\Gamma\mapsto [K;G]$ induces a well-defined linear map
\[ \psi_n:\calA_n^\null(S^1;\Lambda_M)\to \calK_n(M)/\calK_{n+1}(M).\]
If $\pi_1(M)$ is abelian, then $\psi_n$ is surjective\footnote{In \cite{Ha}, Habiro has obtained similar result, showing that the natural surgery map gives a surjection from the space of Jacobi diagrams with $H_1(M)$-colored univalent vertices to the graded quotient in his filtration, without any assumption on $\pi_1(M)$. The techniques used in Proposition~\ref{prop:psi-surjective} are almost the same as those used in Habiro's result.}.
\end{Prop}
\begin{proof}
First, we prove that the assignment $\Gamma\mapsto [K;G]$ gives a well-defined map 
\[ \psi_n^0:\calG_n^\null(S^1;\Lambda_M)\to \calK_n(M)/\calK_{n+1}(M).\]
Let $[K;G]$ and $[K';G']$ be two forest schemes of degree $n$, size $k$ that correspond to a monomial $\Lambda_M$-colored graph $\Gamma\in \calG_n^\null(S^1;\Lambda_M)$, which have only disk-leaves. These forest schemes both belongs to $\calK_n(M)$. Then $[K;G]$ and $[K';G']$ are related by a sequence of the following moves:
\begin{enumerate}
\item A crossing change of knot.
\item A crossing change between edges of tree claspers.
\item A crossing change between an edge of a tree clasper and knot.
\item A bordism change of knot.
\item A bordism change of an edge of a tree clasper.
\item A swapping of a node and a disk-leaf that correspond to trivalent verteces.
\end{enumerate}
In each case, the difference of a change is given by an element of $\calK_{n+1,k+1}(M)$. Indeed, the case (1) is obvious. In the cases (2) and (3), the difference is of the form $[K;G_1\cup C,G_2,\ldots,G_k]-[K;G_1,\ldots,G_k]$, where $C$ is an $I$-clasper whose leaves may link with edges of tree claspers as follows. 
\[ \fig{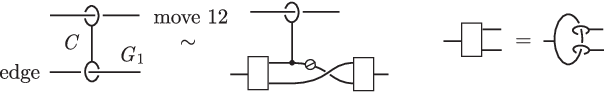} \]
The right hand side is obtained by Habiro's move 12 of \cite[Proposition~2.7]{Ha}. The two boxes can be moved by Habiro's move 11 along the tree claspers toward the univalent ends, so that the component $T$ including the node in the right hand side of the above picture does not have boxes. 
\[ \fig{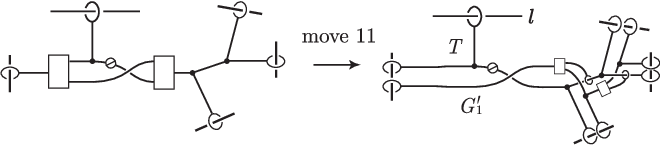} \]
Here if the two leaves of $C$ links with only $G_1$, then the string labeled $l$ may be doubled after the slides of the boxes. By Lemma~\ref{lem:prop-scheme} (3) we have
\[ \begin{split}
  &[K;G_1\cup C,G_2,\ldots,G_k]-[K;G_1,\ldots,G_k]\\
  &=[K;G'_1\cup T,G_2,\ldots,G_k]-[K;G_1,\ldots,G_k]\\
  &=[K;G'_1,G_2,\ldots,G_k]+[K^{G'_1};T,G_2,\ldots,G_k]-[K;G_1,\ldots,G_k]\\
  &=[K^{G'_1};T,G_2,\ldots,G_k]=0\in \calK_n(M)/\calK_{n+1}(M).
\end{split}\]
In the case (4), the difference is of the form $[K^C;G_1,\ldots,G_k]-[K;G_1,\ldots,G_k]=[K;C,G_1,\ldots,G_k]\in\calK_{n+1,k+1}(M)$, where $C$ is an $I$-clasper for $K$ with one disk-leaf and one leaf that is nullhomologous in $M$. In the case (5), the difference is of the form $[K;G_1\cup C,G_2,\ldots,G_k]-[K;G_1,\ldots,G_k]$, where $C$ is an $I$-clasper as follows.
\[ \fig{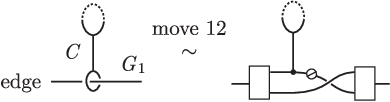} \]
The rest is similar to the case (2) except that a disk-leaf of $C$ is replaced with a nullhomologous leaf. For the case (6), see e.g., \cite[Appendix~E, p.389]{Oh}. 

Next, we shall see that the images of the relations for $\calA_n^\null(S^1;\Lambda_M)$ under $\psi_n^0$ is zero in $\calK_n(M)/\calK_{n+1}(M)$. The proofs that the AS, IHX relations are mapped by $\psi_n^0$ to $\calK_{n+1}(M)$ are similar as in \cite[\S{8.2}]{Ha} or \cite[Theorem~4.11]{GGP}. The STU relation holds by a result of \cite[\S{8.2}]{Ha} (see also \cite[Appendix~E]{Oh}). The FI relation can be checked as follows.
\[ \fig{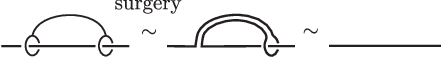} \]
The Orientation Reversal, Linearity and Holonomy relation hold because two embeddings $\rho_1,\rho_2:\Gamma\to M$ are edge-bordant if and only if $\iota\circ\rho_1$ and $\iota\circ\rho_2$ are homotopic, and $\psi_n^0$ is invariant under bordism changes of $\rho$, where we say that two embedded graphs are {\it edge-bordant} if they are related by homotopies and relative bordism changes of edges.

Finally, we shall check the surjectivity of $\psi_n$ when $\pi_1(M)$ is abelian. Let $[K;G]$ be any $M$-null forest scheme in $\calK_n(M)$. By the STU relations, we need only to consider the case where $[K;G]\in\calK_{n,n}(M)$, where $G$ consists of $n$ disjoint collection of $I$-claspers. If $G$ has only disk-leaves, then by Lemma~\ref{lem:prop-scheme} (3) it can be modified modulo $\calK_{n+1}(M)$ to a sum of forest schemes with only simple disk-leaves (a {\it simple} disk-leaf is a disk-leaf that intersects $K$ in a single point transversally. See Figure~\ref{fig:tree-clasper}). Each such forest scheme is clearly obtained by the construction $\psi_n$. If $G$ has an $I$-clasper $C$ with a nullhomologous leaf, then the leaf is nullhomotopic since $\pi_1(M)$ is abelian. Thus the surgery on $C$ can be replaced with a sequence of surgeries on strict $I$-claspers and $[K;G]$ can be rewritten as a sum of forest schemes with only strict $I$-claspers. The rest is the same as above.
\end{proof}

The main theorem of the present paper is the following.
\begin{Thm}\label{thm:injective}
If $H_1(M)=\Z$ and $M$ is fibered over $S^1$, then the following hold.
\begin{enumerate}
\item $\psi_1$ is injective.
\item There is a linear map $Z_2:\calK_2(M)/\calK_3(M)\to\calA_n(S^1;\widehat{\Lambda}_M)$ such that the composition of
\[ \calA_2^\null(S^1;\Lambda_M)\stackrel{\psi_2}{\to}\calK_2(M)/\calK_3(M)\stackrel{Z_2}{\to}\calA_2(S^1;\widehat{\Lambda}_M) \]
agrees with the natural map $\calA_2^\null(S^1;\Lambda_M)\to \calA_2(S^1;\widehat{\Lambda}_M)$.
\end{enumerate}
\end{Thm}
Theorem~\ref{thm:injective} will be proved in \S\ref{s:proof} by using perturbative invariants given in \S\ref{s:perturbative_inv}. Theorem~\ref{thm:injective} already shows that the classification of nullhomologous knots by finite $M$-null type invariants is rather fine. As a corollary to Proposition~\ref{prop:psi-surjective} and Theorem~\ref{thm:injective}, we have the following.
\begin{Cor}[{Part of a result of Lieberum \cite{Lie}}]
If $M=S^2\times S^1$, then the map $\psi_1$ is an isomorphism.
\end{Cor}

\subsection{Degree 1 part}\label{ss:degree_1}

The following remark is an expansion from a comment of K.~Habiro. Since any nullhomologous knot can be unknotted by a surgery on an $M$-null $I$-clasper, we have $\calK_0(M)/\calK_1(M)\cong \Q$. The degree 1 part is more complicated. Let $\varpi'$ be the preimage of $0$ of the natural map $\pi_0(LM)\to H_1(M)$ where $LM$ is the free loop space of $M$ and let $\Q\varpi'$ denote the vector space over $\Q$ spanned by the set $\varpi'$. Let 
\[ \eta:\calK_1(M)/\calK_2(M)\to \Q\varpi'\]
be the linear map that assigns to each forest scheme $[K;G_1]=K^{G_1}-K$ its homotopy class $[K^{G_1}]-[K]$. This is well-defined because any forest scheme of degree 2, size 2 contains a strict $I$-clasper and because surgery on a strict $I$-clasper does not change the free homotopy class of knot. Now we assume that $H_1(M)=\Z$ and that $M$ is fibered over $S^1$. Then by Theorem~\ref{thm:injective} we have the following chain complex, which may be non-exact only at $\calK_1(M)/\calK_2(M)$.
\begin{equation}\label{eq:K1K2}
 0\longrightarrow \calA_1^\null(S^1;\Lambda) 
\stackrel{\psi_1}{\longrightarrow} \calK_1(M)/\calK_2(M)
\stackrel{\eta}{\longrightarrow} \Q\varpi'
\stackrel{\ve}{\longrightarrow} \Q
\longrightarrow 0,
\end{equation}
where $\ve$ is the augmentation map. To make the sequence exact, it may be necessary to consider a ``noncommutative'' refinement of the map $\psi_n$ as in \cite{Ka, KL, Va2} by the resolutions of singular knots (or allowing only strict $I$-claspers) and to restrict underlying knots in the definition of the filtration $\calK_n(M)$ to (possibly base-pointed) nullhomotopic knots in $M$. 

We will see in Appendix~\ref{s:A_1} that there is a natural isomorphism $\calA_1^\null(S^1;\Lambda)\cong \Q[t]/\Q$.

\mysection{Perturbative invariants of nullhomologous knots in $M$}{s:perturbative_inv}

Let $\Lambda=\Q[t^{\pm 1}]$ and $\widehat{\Lambda}=\Q(t)$. In this section, we define a map $Z_n:\calK(M)\to \calA_n(S^1;\widehat{\Lambda})$ for $n=1,2$ (Theorem~\ref{thm:invariance}). Roughly, $Z_n$ is defined by intersections of certain fundamental chains in Lescop's equivariant version of the configuration spaces (\S\ref{ss:e-conf}). We use parametrized Morse theory to give explicit fundamental chains in the equivariant configuration spaces (\S\ref{ss:FMF}, \S\ref{ss:Z-path}, \S\ref{ss:e-propagator}). It turns out that $Z_n$ can be interpreted as the trace of a generating function of counts of certain graphs in $M$, which we call Z-graphs (\S\ref{ss:moduli_hori}). 

\subsection{Equivariant configuration space}\label{ss:e-conf}

From now on we consider an oriented surface bundle $\kappa:M\to S^1$. Let $K:S^1\to M$ be a nullhomologous knot. We shall define a configuration space that is suitable to our purpose, based on Lescop's equivariant configuration space \cite{Les2, Les3, Les4}.

Let $\bConf_q(X)$ denote the Fulton--MacPherson--Kontsevich compactification of the configuration space $\Conf_q(X)$ of $q$ distinct points on a compact differentiable (real) manifold $X$ (see \cite{Ko1, FM, BT, Les1} etc. for detail). Let $\bConf_q(S^1)_{<}$ denote the component of $\bConf_q(S^1)$ for a fixed ordering of the points on $S^1$. Let $\bConf_{t,q}(M,K)$ be the pullback in the following commutative diagram:
\[ \xymatrix{
  \bConf_{t,q}(M,K) \ar[r] \ar[d] & \bConf_{t+q}(M) \ar[d]^-{\pi_{t,q}}\\
  \bConf_q(S^1)_{<} \ar[r]_-{C_q^K} & \bConf_q(M)
}\]
where $\pi_{t,q}$ is the natural map associated to the projection $\Conf_{t+q}(M)\to \Conf_q(M)$ and $C_q^K$ is the smooth extension of $K\times\cdots\times K:\Conf_q(S^1)_<\to \Conf_q(M)$. 

Let  $\Gamma$ be a labeled Jacobi diagram with $q$ univalent and $t$ trivalent vertices. By the labeling $\alpha:\{1,2,\ldots,t+q\}\to V(\Gamma)$, we identify $E(\Gamma)$ with the set of ordered pairs $(i,j)$, $i,j\in \{1,2,\ldots,t+q\}$. Let $M^\Gamma$ denote the set of tuples
\[ (x_1,x_2,\ldots,x_{2n},\{\gamma_{ij}\}_{(i,j)\in E(\Gamma)}), \]
where $x_i\in M$ and $\gamma_{ij}$ is the homotopy class of continuous maps $c_{ij}:[0,1]\to S^1$ relative to the endpoints such that $c_{ij}(0)=\kappa(x_i)$ and $c_{ij}(1)=\kappa(x_j)$. We consider $M^\Gamma$ as a topological space as follows. Let $C^0(\Gamma,S^1)$ be the space of continuous maps $\Gamma\to S^1$ equipped with the $C^0$-topology and let $C^\Gamma$ be the space that is the pullback in the following commutative diagram.
\[ \xymatrix{
  C^\Gamma \ar[r] \ar[d] & C^0(\Gamma,S^1) \ar[d]\\
  M^{t+q} \ar[r]^-{\kappa\times\cdots\times\kappa} & (S^1)^{t+q}
}\]
The fiberwise quotient map $C^\Gamma\to M^\Gamma$ by the homotopy relation of edges gives $M^\Gamma$ the quotient topology. 

Let $\bConf_\Gamma(M)$ denote the space obtained from $M^\Gamma$ by blowing-up along all the lifts of the diagonals in $M^{t+q}$. Let $\bConf_\Gamma(M,K)$ denote the pullback in the following commutative diagram:
\[ \xymatrix{
  \bConf_\Gamma(M,K) \ar[r] \ar[d] & \bConf_\Gamma(M) \ar[d]\\
  \bConf_q(S^1)_< \ar[r]_-{C_O^K} & \bConf_O(M)
}\]
where $O$ is the subgraph of $\Gamma$ given by the Wilson loop and $C_O^K$ is the natural map induced by $K\times\cdots\times K$, i.e., $C_q^K$ together with the relative homotopy classes of arcs represented by those in $K$. The forgetful map $\overline{\pi}:\bConf_\Gamma(M,K)\to \bConf_{t,q}(M,K)$ is a $\Z^{3n-q}$-covering ($3n-q=|E^{nW}(\Gamma)|$). Since $\bConf_{\Gamma}(M,K)$ is naturally a $\Z^{3n-q}$-space by the covering translation, the twisted homology
\[ H_*(\bConf_{\Gamma}(M,K);\Q)\otimes_{\Lambda_{\Gamma}}\widehat{\Lambda}_{\Gamma}, \]
where $\Lambda_{\Gamma}=\Q[\{t_{ij}^{\pm 1}\}_{(i,j)\in E^{nW}(\Gamma)}]$ and $\widehat{\Lambda}_{\Gamma}=\bigotimes_{(i,j)\in E^{nW}(\Gamma)} \Q(t_{ij})$ (tensor product of $\Q$-modules), is defined. Here, $\Lambda_{\Gamma}$ acts on $\widehat{\Lambda}_{\Gamma}$ by $(\prod_{(i,j)}t_{ij}^{k_{ij}})(\bigotimes_{(i,j)}f_{ij}) = \bigotimes_{(i,j)}t_{ij}^{k_{ij}}f_{ij}$. In fact, the $\Gamma$ in the definition of $\bConf_\Gamma(M,K)$ could be any graph. We will consider the interval $K_2$ for $\Gamma$.

\subsection{Fiberwise Morse functions and their concordances}\label{ss:FMF}

Let $\kappa:M\to S^1$ be a smooth fiber bundle with fiber diffeomorphic to a closed connected oriented 2-manifold $\Sigma$. We equip $M$ with a Riemannian metric. We fix a fiberwise Morse function $f:M\to \R$ and its gradient $\xi:M\to \mathrm{Ker}\,d\kappa$ along the fibers that satisfies the parametrized Morse--Smale condition, i.e., the descending manifold loci and the ascending manifold loci are mutually transversal in $M$. We consider only fiberwise Morse functions that are {\it oriented}, i.e., the bundles of negative eigenspaces of the Hessians along the fibers on the critical loci are oriented. There always exists an oriented fiberwise Morse function on $M$ (e.g., \cite{Wa2}). 

A {\it generalized Morse function (GMF)} is a $C^\infty$ function on a manifold with only Morse or birth-death singularities (\cite[Appendix]{Ig}). A {\it fiberwise GMF} for a fiber bundle $\kappa:E\to B$ is a $C^\infty$ function $f:E\to \R$ whose restriction $f_c=f|_{\kappa^{-1}(c)}:\kappa^{-1}(c)\to \R$ is a GMF for all $c\in B$. A {\it critical locus} of a fiberwise GMF is the subset of $E$ consisting of critical points of $f_c$, $c\in B$. A fiberwise GMF is {\it oriented} if it is oriented outside birth-death loci and if birth-death pairs near a birth-death locus have incidence number 1.

It is known (e.g. Framed Function Theorem of \cite{Ig}) that for a pair of fiberwise Morse functions $f_0,f_1:M\to \R$, there exists a homotopy $\widetilde{f}=\{f_s\}_{s\in [0,1]}$ between $f_0$ and $f_1$ in the space of oriented GMF's on $M$, which gives an oriented fiberwise GMF on the surface bundle $\kappa\times\mathrm{id}:M\times [0,1]\to S^1\times [0,1]$. 

\begin{Def}
We say that the homotopy $\widetilde{f}$ is a {\it concordance} if each birth-death locus of $\wf$ in $M\times [0,1]$ projects by $\kappa\times\mathrm{id}$ to a simple closed curve that is not nullhomotopic in $S^1\times[0,1]$. 
\end{Def}

\subsection{Z-paths}\label{ss:Z-path}

Let $\pi:\wM\to M$ be the $\Z$-covering associated to $\kappa$. Let $\widetilde{\kappa}:\widetilde{M}\to \R$ be the lift of $\kappa$. Let $\wf:\wM\to \R$ denote the $\Z$-invarint lift $\wf=f\circ \pi$ and let $\widetilde{\xi}$ denote the lift of $\xi$. We say that a piecewise smooth embedding $\sigma:[\mu,\nu]\to \wM$ is {\it vertical} if $\mathrm{Im}\,\sigma$ is included in a single fiber of $\widetilde{\kappa}$ and say that $\sigma$ is {\it horizontal} if $\mathrm{Im}\,\sigma$ is included in a critical locus of $\wf$. We say that a vertical embedding (resp. horizontal embedding) $\sigma:[\mu,\nu]\to \wM$ is {\it descending} if $\wf(\sigma(\mu))\geq  \wf(\sigma(\nu))$ (resp. $\widetilde{\kappa}(\sigma(\mu))\leq \widetilde{\kappa}(\sigma(\nu))$). 

A {\it flow-line of $-\widetilde{\xi}$} is a vertical smooth embedding $\sigma:[\mu,\nu]\to \wM$ such that for each $T\in[\mu,\nu]$ that is not in the preimage of the union of critical loci, $d\sigma_T(\frac{\partial}{\partial T})$ is a multiple of $(-\widetilde{\xi})_{\sigma(T)}$ by a positive real number. 

\begin{Def}\label{def:al-path}
Let $x,y$ be two points of $\widetilde{M}$ such that $\widetilde{\kappa}(x)\leq\widetilde{\kappa}(y)$. A {\it Z-path from $x$ to $y$} is a sequence $\gamma=(\sigma_1,\sigma_2,\ldots,\sigma_n)$, $n\geq 1$, where
\begin{enumerate}
\item For each $i$, $\sigma_i$ is either vertical or horizontal.
\item For each $i$, $\sigma_i$ is a descending embedding $[\mu_i,\nu_i]\to \widetilde{M}$ for some real numbers $\mu_i,\nu_i$.
\item If $\sigma_i$ is vertical, then $\sigma_i$ is a flow line of $-\widetilde{\xi}$. If $\sigma_i$ is horizontal, then $\mu_i<\nu_i$.
\item $\sigma_1(\mu_1)=x$, $\sigma_n(\nu_n)=y$.
\item $\sigma_i(\nu_i)=\sigma_{i+1}(\mu_{i+1})$ for $1\leq i<n$.
\item If $\sigma_i$ is vertical (resp. horizontal) and if $i<n$, then $\sigma_{i+1}$ is horizontal (resp. vertical).
\item If $n=1$, then $\mu_1<\nu_1$.
\end{enumerate}
We say that two Z-paths are {\it equivalent} if they differ only by reparametrizations on segments. A Z-path in $M$ is defined as the composition of a Z-path in $\wM$ with the covering projection $\pi$.
\end{Def}

For generic $\xi$, there may be special vertical flow-line between critical loci, called {\it $1/1$-intersection}, which is the transversal intersection of the descending manifold locus of a critical locus of $\xi$ of index 1 and the ascending manifold locus of another critical locus of index 1. There may be finitely many $1/1$-intersections for generic $\xi$. Most of the vertical segments in Z-paths are $1/1$-intersections.

\subsection{Equivariant propagator}\label{ss:e-propagator}

Let $\xi$ be the fiberwise gradient of an oriented fiberwise Morse function on $M$. We say that a nonconstant Z-path $\gamma$ in $M$ with positive length is a {\it closed Z-path} if the endpoints of $\gamma$ coincide. A closed Z-path $\gamma$ gives a piecewise smooth map $\bar{\gamma}:S^1\to M$, which can be considered as a ``closed orbit'' in $M$. We will also call $\bar{\gamma}$ a closed Z-path. A closed Z-path has an orientation that is determined by the orientations of intersections of the loci of descending and ascending manifolds of $\widetilde{\xi}$. See \cite[\S{2.7}]{Wa3} for the detail. Then we define the sign $\ve(\gamma)\in\{-1,1\}$ and the period $p(\gamma)$ of $\gamma$ by 
\[ p(\gamma)=|\langle[d\kappa],[\bar{\gamma}]\rangle|,\quad 
\ve(\gamma)=\frac{\langle[d\kappa],[\bar{\gamma}]\rangle}{|\langle[d\kappa],[\bar{\gamma}]\rangle|}.
\]
Let $S(TM)$ be the oriented subbundle of $TM$ of unit tangent vectors. Let $S(T\gamma)$ be the pullback $\bar{\gamma}^*S(TM)$, which can be considered as a piecewise smooth 3-dimensional chain in $\partial\bConf_{K_2}(M)$. We say that two closed Z-paths $\gamma_1$ and $\gamma_2$ are {\it equivalent} if there is a degree 1 homeomorphism $g:S^1\to S^1$ such that $\bar{\gamma}_1\circ g=\bar{\gamma}_2$. The indices of horizontal segments in a closed Z-path must be all equal since a Z-path is descending. We define the index $\mathrm{ind}\,\gamma$ of a closed Z-path $\gamma$ to be the index of a horizontal segment (critical locus) in $\gamma$, namely, the index of the critical point of $f|_{\kappa^{-1}(c)}$ for any $c\in S^1$ that is the intersection of $\gamma$ with $\kappa^{-1}(c)$. 

Let $M_0=M\setminus\bigcup_{\gamma\,:\,\mathrm{critical\, locus}}\gamma$ and let $s_\xi:M_0\to S(TM_0)$ be the normalization $-\xi/\|\xi\|$ of the section $-\xi$. The closure $\overline{s_\xi(M_0)}$ in $S(TM)$ is a smooth manifold with boundary whose boundary is the disjoint union of circle bundles over the critical loci $\gamma$ of $\xi$, for a similar reason as \cite[Lemma~4.3]{Sh}. The fibers of the circle bundles are equators of the fibers of $S(T\gamma)$. Let $E^-_\gamma$ be the total space of the 2-disk bundle over $\gamma$ whose fibers are the lower hemispheres of the fibers of $S(T\gamma)$ which lie below the tangent spaces of the level surfaces of $\kappa$. Then $\partial\overline{s_\xi(M_0)}=\bigcup_\gamma \partial E_\gamma^-$ as sets. Let
\[ s_\xi^*(M)=\overline{s_\xi(M_0)}\cup \bigcup_\gamma E^-_\gamma\subset S(TM).\]
This is a 3-dimensional piecewise smooth manifold. We orient $s_\xi^*(M)$ by extending the natural orientation $(s_\xi^{-1})^*o(M)$ on $s_\xi(M_0)$ induced from the orientation $o(M)$ of $M$. The piecewise smooth projection $s_\xi^*(M)\to M$ is a homotopy equivalence and $s_\xi^*(M)$ is homotopic to $s_{\hat{\xi}}$. 

Let $C$ be a knot in $M$ such that $\langle[d\kappa],[C]\rangle=1$. Let $\acalM_{K_2}(\xi)$ be the set of all Z-paths in $M$. There is a natural structure of non-compact manifold with corners on $\acalM_{K_2}(\xi)$. For a closed Z-path $\gamma$, we denote by $\gamma^\irr$ the minimal closed Z-path such that $\gamma$ is equivalent to the iteration $(\gamma^\irr)^k$ for a positive integer $k$ and we call $\gamma^\irr$ the {\it irreducible} factor of $\gamma$. This is unique up to equivalence. If $\gamma=\gamma^\irr$, we say that $\gamma$ is irreducible. We orient $S(T\gamma^\irr)$ so that $[S(T\gamma^\irr)]=p(\gamma^\irr)[S(TC)]$. Note that this may not be the one naturally induced from the orientation of $\gamma^\irr$ but from $\ve(\gamma^\irr)\gamma^\irr$. 

\begin{Thm}[\cite{Wa2}]\label{thm:propagator}
Let $M$ be the mapping torus of an orientation preserving diffeomorphism $\varphi:\Sigma\to \Sigma$ of closed, connected, oriented surface $\Sigma$. Let $\xi$ be the fiberwise gradient of an oriented fiberwise Morse function $f:M\to \R$. 
\begin{enumerate}
\item There is a natural closure $\bacalM_{K_2}(\xi)$ of $\acalM_{K_2}(\xi)$ that has the structure of a countable union of smooth compact manifolds with corners.
\item Let $\bar{b}:\bacalM_{K_2}(\xi)\to M^{K_2}$ be the evaluation map, which assigns the pair of the endpoints of a Z-path $\gamma$ together with the homotopy class of $\kappa\circ \gamma$ relative to the endpoints. Let $B\ell_{\bar{b}^{-1}(\widetilde{\Delta}_M)}(\bacalM_{K_2}(\xi))$ denote the blow-up of $\bacalM_{K_2}(\xi)$ along $\bar{b}^{-1}(\widetilde{\Delta}_M)$. Then $\bar{b}$ induces a map $B\ell_{\bar{b}^{-1}(\widetilde{\Delta}_M)}(\bacalM_{K_2}(\xi))\to \bConf_{K_2}(M)$ and it represents a 4-dimensional $\widehat{\Lambda}$-chain $Q(\xi)$ in $\bConf_{K_2}(M)$ that satisfies the identity
\[ \partial Q(\xi)=s_{\xi}^*(M)+\sum_\gamma (-1)^{\mathrm{ind}\,\gamma}\ve(\gamma)\,t^{p(\gamma)}\,S(T\gamma^\irr), \]
where the sum is taken over equivalence classes of closed Z-paths in $M$. Moreover, for a product $C(t)$ of cyclotomic polynomials, $C(t)\det(1-t\varphi_{*1})Q(\xi)$ is a $\Lambda$-chain, where $\varphi_{*1}:H_1(\Sigma;\Q)\to H_1(\Sigma;\Q)$ is the map induced by $\varphi$. 
\end{enumerate}
\end{Thm}

\subsection{Equivariant intersection form}\label{ss:count}

Let $\kappa:M\to S^1$ be a fibration as above. Fix a closed connected oriented 2-submanifold $\Sigma$ of $M$ such that the oriented bordism class of $\Sigma$ in $M$ corresponds to $[\kappa]$ via the canonical isomorphism $\Omega_2(M)=H_2(M)\cong H^1(M)$. Let $\Gamma$ be a labeled oriented Jacobi diagram of degree $n$ and let $E^I(\Gamma)$ (resp. $E^\rho(\Gamma)$) be the subset of $E^{nW}(\Gamma)$ consisting of non self-loop edges (resp. self-loop edges).
\begin{enumerate}
\item If $e\in E^I(\Gamma)$, then let $\psi_e:\bConf_{\Gamma}(M,K)\to \bConf_{K_2}(M)$ denote the projection that gives the endpoints of $e$ together with the associated path in $S^1$. Take a compact oriented 4-submanifold $F_e$ in $\bConf_{K_2}(M)$ with corners.
\item If $e\in E^\rho(\Gamma)$, then let $\psi_e:\bConf_{\Gamma}(M,K)\to M$ denote the projection that gives the unique endpoint of $e$. Take a compact oriented 1-submanifold $F_e$ in $M$ with boundary.
\end{enumerate}
Note that in both cases $F_i$ is of codimension 2. Now we put
\[\begin{split}
&E^{nW}(\Gamma)=\{i_1,\ldots,i_{3n-q}\},\quad E^W(\Gamma)=\{j_1,\ldots,j_q\},\\
&V^{nW}(\Gamma)=\{k_1,\ldots,k_{2n-q}\},\quad V^W(\Gamma)=\{\ell_1,\ldots,\ell_q\}.
\end{split}\]
Then we define
\[ \langle F_{i_1},F_{i_2},\ldots,F_{i_{3n-q}}\rangle_\Gamma
=\bigcap_{e=1}^{3n-q} \psi_{i_e}^{-1}(F_{i_e}), \]
which gives a compact 0-dimensional submanifold in $\bConf_\Gamma(M,K)$ if the intersection is transversal. We equip each point $(u_{\ell_1},\ldots,u_{\ell_q},v_{k_1},\ldots,v_{k_{2n-q}};\gamma_1,\gamma_2,\ldots,\gamma_{3n})$ of $\langle F_{i_1},F_{i_2},\ldots,F_{i_{3n-q}}\rangle_\Gamma$ with a coorientation (a sign) in $\bConf_{\Gamma}(M,K)$ by
\[ \bigwedge_{e\in E^{nW}(\Gamma)} \psi_e^*\,o^*(F_e)\in \twedge^{6n-2q} T_{(u_{\ell_1},\ldots,u_{\ell_q},v_{k_1},\ldots,v_{k_{2n-q}})}\bConf_{t,q}(M,K). \]
Here, we identify a neighborhood of a point in $\bConf_\Gamma(M,K)$ with its image of the covering projection in $\bConf_{t,q}(M,K)$. The coorientation gives a sign as the sign of $\mu$ in the equation $\bigwedge_e\psi_e^*o^*(F_e)=\mu\, o(S^1)_{u_{\ell_1}}\wedge\cdots\wedge o(S^1)_{u_{\ell_q}}\wedge o(M)_{v_{k_1}}\wedge\cdots\wedge o(M)_{v_{k_{2n-q}}}$, where the order of the product is determined by the vertex labeling, namely, so that it agrees with the exterior product of $o(S^1)$ and $o(M)$ in the order of the vertex labeling. By this, $\langle F_{i_1},F_{i_2},\ldots,F_{i_{3n-q}}\rangle_\Gamma$ represents a 0-chain in $\bConf_{\Gamma}(M,K)$. This can be extended to generic tuples of codimension 2 $\Q$-chains by multilinearity. We will denote the homology class of $\langle F_{i_1},F_{i_2},\ldots,F_{i_{3n-q}}\rangle_\Gamma$ (integer) by the same notation. Note that a point of $\bConf_\Gamma(M,K)$ possesses canonical homotopy classes of edges of $E^W(\Gamma)$ in $S^1$ determined by the embedding $C_O^K:\bConf_q(S^1)_<\to \bConf_O(M)$.

We extend the form $\langle\cdot,\ldots,\cdot\rangle_\Gamma$ to tuples of codimemsion 2 $\Lambda$-chains $F_k'$ in $\bConf_{K_2}(M)$ or $M$ as follows. When $k\in E^I(\Gamma)$, suppose that $F_k'$ is of the form $\sum_{\lambda_k=1}^{N_k} \mu_{\lambda_k}^{(k)} \sigma_{\lambda_k}^{(k)}$, where $\mu_{\lambda_k}^{(k)}\in \Lambda$ and $\sigma_{\lambda_k}^{(k)}$ is a compact oriented smooth 4-submanifold in $\bConf_{K_2}(M)[0]$, where $\bConf_{K_2}(M)[0]$ is the subspace of $\bConf_{K_2}(M)$ consisting of $(x_1,x_2,\gamma_{12})$ such that $\gamma_{12}$ has a lift $\sigma_{12}:[0,1]\to M$ connecting $x_1$ and $x_2$ with $[\kappa\circ\sigma_{12}]=\gamma_{12}$ relative to the boundary whose interior has algebraic intersection number $0$ with $\Sigma$. When $k\in E^\rho(\Gamma)$, suppose that $F_k'$ is of the form $\sum_{\lambda_k=1}^{N_k} \mu_{\lambda_k}^{(k)} \sigma_{\lambda_k}^{(k)}$, where $\mu_{\lambda_k}^{(k)}\in\Lambda$ and $\sigma_{\lambda_k}^{(k)}$ is a piecewise smooth path in $M$. Then we define
\[ \begin{split}
& \langle F_{i_1}',F_{i_2}',\ldots,F_{i_{3n-q}}'\rangle_\Gamma\\
&=\sum_{\lambda_1,\lambda_2,\ldots,\lambda_{3n-q}}
\mu_{\lambda_1}^{(1)}(t_{i_1})\mu_{\lambda_2}^{(2)}(t_{i_2})\cdots\mu_{\lambda_{3n-q}}^{(3n-q)}(t_{i_{3n-q}})
\bigl\langle \sigma_{\lambda_1}^{(1)},\sigma_{\lambda_2}^{(2)},\ldots,\sigma_{\lambda_{3n-q}}^{(3n-q)}\bigr\rangle_\Gamma\\
&\in C_0(\bConf_{\Gamma}(M,K);\Q),
\end{split} \]
which can be considered as a 0-chain in $\bConf_{t,q}(M,K)$ with coefficients in $\Lambda_\Gamma$. This is multilinear by definition. 

Next, we extend the form $\langle\cdot,\ldots,\cdot\rangle_\Gamma$ to tuples of codimension 2 $\widehat{\Lambda}$-chains in $\bConf_{K_2}(M)$ or $M$ as follows. Let $Q_{i_1},Q_{i_2},\ldots,Q_{i_{3n-q}}$ be codimension 2 $\widehat{\Lambda}$-chains in $\bConf_{K_2}(M)$ or $M$ depending on whether the corresponding edge is not a self-loop or a self-loop. Then there exist Laurent polynomials $p_1,p_2,\ldots,p_{3n-q}\in \Lambda\setminus\{0\}$ such that $p_kQ_{i_k}$ is a $\Lambda$-chain for each $k$. We define
\[ \begin{split}
&  \langle Q_{i_1},Q_{i_2},\ldots,Q_{i_{3n-q}}\rangle_\Gamma\\
&=\langle p_1Q_{i_1},p_2Q_{i_2},\ldots,p_{3n-q}Q_{i_{3n-q}}\rangle_\Gamma
\,p_1(t_{i_1})^{-1} p_2(t_{i_2})^{-1}\cdots  p_{3n-q}(t_{i_{3n-q}})^{-1}\\
&\in C_0(\bConf_{\Gamma}(M,K);\Q)\otimes_{\Lambda_\Gamma}\widehat{\Lambda}_\Gamma,
\end{split} \]
which can be considered as a 0-chain in $\bConf_{t,q}(M,K)$ with coefficients in $\widehat{\Lambda}_\Gamma$. This does not depend on the choices of $p_1,\ldots,p_{3n-q}$ and this is $\widehat{\Lambda}$-multilinear by definition. Note that the multilinear form $\langle\cdot,\ldots,\cdot\rangle_\Gamma$ depends on the choice of $\Sigma$. 

We define a linear map 
\[ \Tr_\Gamma:C_0(\bConf_\Gamma(M,K))\otimes_{\Lambda_\Gamma}\widehat{\Lambda}_\Gamma \to \calA_n(S^1;\widehat{\Lambda}) \]
by taking the homology class (rational number) in $H_0(\bConf_{2n}(M);\Q)$ on the chain side and by giving the coloring of $\Gamma$ from the coefficient and the $\Lambda$-holonomies on edges in $E^W(\Gamma)$.

\subsection{Definition of $Z_n$}\label{ss:def_Z}

Let $\kappa_1,\kappa_2,\ldots,\kappa_{3n}:M\to S^1$ be fibrations isotopic to $\kappa$. Let $f_i:M\to \R$, $i=1,2,\ldots,3n$, be oriented fiberwise Morse functions for $\kappa_i$ such that $(\kappa_i,f_i)$ is concordant to a pair isotopic to $(\kappa,f)$. Let $\xi_i$ be the fiberwise gradient of $f_i$. Let $Q(\xi_i)$ be the equivariant propagator in Theorem~\ref{thm:propagator} for $\xi_i$. We define the 1-cycle
\[ Q'(\xi_i)=\sum_\gamma (-1)^{\mathrm{ind}\,\gamma}\ve(\gamma)\,t^{p(\gamma)}\gamma^\irr\in C_1(M;\Q)\otimes_\Q\widehat{\Lambda} \]
in $M$, where the sum is over equivalence classes of lifts of all closed Z-paths $\gamma$ for $\xi_i$ considered as oriented 1-cycles. This is an infinite sum but is well-defined as a $\widehat{\Lambda}$-chain. The orientation of a closed Z-path $\gamma$ is given by $\ve(\gamma)$ times the downward orientation on $\gamma$. Choosing the closed oriented surface $\Sigma\subset M$ and $\kappa_i,\xi_i$ generically, we may define 
\[ I_\Gamma(K)=\Tr_\Gamma\langle Q^\circ(\xi_{i_1}),Q^\circ(\xi_{i_2}),\ldots,Q^\circ(\xi_{i_{3n-q}}) \rangle_\Gamma\in \calA_n(S^1;\widehat{\Lambda}), \]
where $Q^\circ(\xi_i)$ is $Q(\xi_i)$ or $Q'(\xi_i)$ depending on whether the corresponding edge in $\Gamma$ is not a self-loop or a self-loop. 
\begin{Thm}\label{thm:invariance}
For $n\geq 1$, we define
\[ Z_n(K)=\frac{1}{2^n(2n)!(3n)!}\sum_{\Gamma}I_\Gamma(K)\in \calA_n(S^1;\widehat{\Lambda}), \]
where the sum is over all labeled Jacobi diagrams on $S^1$ of degree $n$ for all possible edge orientations (Jacobi diagrams of degree 1 and 2 are shown in Figure~\ref{fig:chord-diags}.). For $n=1,2$, $Z_n(K)$ is invariant under isotopy of $K$ and concordance of $f$. 
\end{Thm}
\begin{figure}
\fig{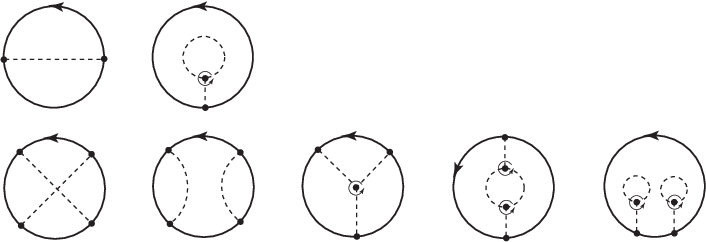}
\caption{Jacobi diagrams of degree 1 and 2 (edge-orientations omitted).}\label{fig:chord-diags}
\end{figure}
Let $\calE(\kappa)$ be the set of concordance classes of oriented fiberwise Morse functions for a fibration $\kappa:M\to S^1$. Theorem~\ref{thm:invariance} says that for $n=1,2$, $Z_n$ gives a family of knot invariants parametrized by $\calE(\kappa)$.  

\subsection{$Z_n$ and the generating function of counts of Z-graphs}\label{ss:moduli_hori}

Let $f_1,f_2,\ldots,f_{3n}$ and $\xi_1,\xi_2,\ldots,\xi_{3n}$ be as in \S\ref{ss:def_Z}. Let $K$ be a nullhomologous knot in $M$.
\begin{Def}
Let $\Sigma=\kappa^{-1}(0)$. Suppose that no $1/1$-intersection curves for $\xi_i$ intersects $\Sigma$. For a labeled Jacobi diagram $\Gamma$ of degree $n$ and for $\vec{k}=(k_1,k_2,\ldots,k_{3n})\in\Z^{3n}$, we define $\acalM_{\Gamma(\vec{k})}(\Sigma;\xi_1,\xi_2,\ldots,\xi_{3n})$ as the set of piecewise smooth maps $I:\Gamma\to M$ such that
\begin{enumerate}
\item if $i\in E^{nW}(\Gamma)$, the restriction of $I$ to the $i$-th edge is a Z-path of $\xi_i$, 
\item if $i\in E^W(\Gamma)$, the restriction of $I$ to the $i$-th edge is embedded to a segment in $K$ in an orientation preserving way,
\item the algebraic intersection number of the restriction of $I$ to the $i$-th edge $e_i$ with $\Sigma$ is $k_i$.
\end{enumerate}
We call such maps {\it Z-graphs} for $(\Sigma;\xi_1,\xi_2,\ldots,\xi_{3n})$ of type $\vec{k}$. We define a topology on $\acalM_{\Gamma(\vec{k})}(\Sigma;\xi_1,\xi_2,\ldots,\xi_{3n})$ as the transversal intersection of smooth submanifolds of $\bConf_\Gamma(M)$, as in \S\ref{ss:count}. (See \cite[Lemma~3.7]{Wa3} for the reason of transversality of $\acalM_{\Gamma(\vec{k})}(\Sigma;\xi_1,\xi_2,\ldots,\xi_{3n})$ for generic choice of $(\kappa_i,f_i)$.)
\end{Def}
\begin{figure}
\fig{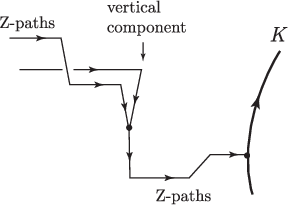}
\caption{A Z-graph for a knot}\label{fig:Z-graph}
\end{figure}
We may identify a point of $\acalM_{\Gamma(\vec{k})}(\Sigma;\xi_1,\ldots,\xi_{3n})$ with an oriented 0-manifold in $\bConf_{\Gamma}(M,K)$. Hence the moduli space $\acalM_{\Gamma(\vec{k})}(\Sigma;\xi_1,\ldots,\xi_{3n})$ can be counted with signs. The sum of signs agrees with the sum of coefficients of the terms of $t_1^{k_1}t_2^{k_2}\cdots t_{3n}^{k_{3n}}$ in the power series expansion of $\langle Q^\circ(\xi_{i_1}),Q^\circ(\xi_{i_2}),\ldots,Q^\circ(\xi_{i_{3n-q}})\rangle_\Gamma$. We denote the sum of signs by $\#\acalM_{\Gamma(\vec{k})}(\Sigma;\xi_1,\ldots,\xi_{3n})$. 

For generic choices of $\Sigma,\kappa_1,\ldots,\kappa_{3n},\xi_1,\ldots,\xi_{3n}$, a Z-graph $I\in \acalM_{\Gamma(\vec{k})}(\Sigma;\xi_1,\ldots,\xi_{3n})$ consists of finitely many vertical components, namely subgraphs consisting of vertical segments, and some Z-paths that connect two univalent vertices of vertical components or one univalent vertex of a vertical component and the knot $K$ (see Figure~\ref{fig:Z-graph}). If a Z-path that starts at (resp. end at) a point of $K$ has a horizontal segment, then its first (resp. last) vertical segment $\sigma$ is the transversal intersection of $K$ and the ascending (resp. descending) manifold locus of a horizontal segment next to $\sigma$ (resp. previous to $\sigma$). The following proposition follows from Theorem~\ref{thm:propagator} and from definition of $Z_n$. 

\begin{Prop}\label{prop:F} For generic choices of $\xi_1,\xi_2,\ldots,\xi_{3n}$ and for a labeled trivalent graph $\Gamma$, let $F_\Gamma(\Sigma;\xi_1,\xi_2,\ldots,\xi_{3n})$ be the generating function
\[ \sum_{\vec{k}=(k_1,\ldots,k_{3n})\in\Z^{3n}}\#\acalM_{\Gamma(\vec{k})}(\Sigma;\xi_1,\xi_2,\ldots,\xi_{3n})\,t_1^{k_1}t_2^{k_2}\cdots t_{3n}^{k_{3n}}, \]
where $\#\acalM_{\Gamma(\vec{k})}(\Sigma;\xi_1,\xi_2,\ldots,\xi_{3n})$ is the count of Z-graphs of type $\vec{k}$ of generic type. Then there exists an element $P(t_1,\ldots,t_{3n})\in \Q[t_1^{\pm 1},\ldots,t_{3n}^{\pm 1}]$ such that
\[ F_\Gamma(\Sigma;\xi_1,\xi_2,\ldots,\xi_{3n})
=P(t_1,\ldots,t_{3n})
\prod_{k=1}^{3n-q}C(t_{i_k})^{-1}\det(1-t_{i_k}\varphi_{*1})^{-1}\]
for a product $C(t)\in\Lambda$ of cyclotomic polynomials. Considering this as an element of $\Q(t_1)\otimes_\Q\Q(t_2)\otimes_\Q\cdots\otimes_\Q\Q(t_{3n})$, we have
\[ Z_n(K) =\frac{1}{2^n(2n)!(3n)!}\sum_\Gamma \Tr_\Gamma F_\Gamma(\Sigma;\xi_1,\xi_2,\ldots,\xi_{3n}). \]
\end{Prop}

\mysection{Invariance of $Z_n$ (proof of Theorem~\ref{thm:invariance})}{s:***}

\subsection{1-parameter families of configuration spaces}

Let $I=[0,1]$. Let $K_0,K_1:S^1\to M$ be two nullhomologous knots that are isotopic. Let $\widetilde{K}:S^1\times I\to M$ be an isotopy between $K_0$ and $K_1$ and for $s\in I$ let $K_s:S^1\to M$ be the embedding $K_s(x)=\widetilde{K}(x,s)$. Let $\bConf_{t,q}(M\times I,\widetilde{K})$ be the $\bConf_{t,q}(M,K_0)$-bundle over $I$ given by the pullback in the following commutative diagram
\[ \xymatrix{
  \bConf_{t,q}(M\times I,\widetilde{K}) \ar[r] \ar[d] & \bConf_{t+q}(M) \ar[d]\\
  \bConf_q(S^1)_<\times I \ar[r]_-{C_q^{\widetilde{K}}} & \bConf_q(M)
}\]
where $C_q^{\widetilde{K}}$ is the natural map induced by $\widetilde{K}$. Let $\Gamma$ be a labeled Jacobi diagram with $t$ trivalent and $q$ univalent vertices. Let $\bConf_\Gamma(M\times I,\widetilde{K})$ be the pullback in the following commutative diagram
\[ \xymatrix{
  \bConf_\Gamma(M\times I,\widetilde{K}) \ar[r] \ar[d] & \bConf_\Gamma(M) \ar[d]\\
  \bConf_q(S^1)_<\times I \ar[r]_-{C_O^{\widetilde{K}}} & \bConf_O(M)
}\]
where $C_O^{\widetilde{K}}$ is the natural map induced by $\widetilde{K}$. Then $\bConf_\Gamma(M\times I,\widetilde{K})$ forms a fiber bundle over $I$ with fiber diffeomorphic to $\bConf_\Gamma(M,K_0)$.

\subsection{Bifurcations of Z-graphs in 1-parameter family, $\xi_i$ fixed}

By replacing $Q^\circ(\xi_i)$ with the chain $\widetilde{Q}^\circ(\xi_i)=Q^\circ(\xi_i)\times I$ of $\bConf_{K_2}(M)\times I$ in the definition of $I_\Gamma(K)$, we may define 
\[ I_\Gamma(\widetilde{K})=\Tr_\Gamma\langle \widetilde{Q}^\circ(\xi_{i_1}),\ldots, \widetilde{Q}^\circ(\xi_{i_{3n-q}})\rangle_\Gamma, \]
where $\langle \widetilde{Q}^\circ(\xi_{i_1}),\ldots, \widetilde{Q}^\circ(\xi_{i_{3n-q}})\rangle_\Gamma$ is the equivariant intersection in $\bConf_\Gamma(M\times I,\widetilde{K})$ defined by fixing $\Sigma$. If $\widetilde{K}$ is generic, this gives a piecewise smooth 1-chain in $C_1(\bConf_{t,q}(M\times I,\widetilde{K});\Q)\otimes_\Q\calA_n(S^1;\widehat{\Lambda})$. A bordism change of $\Sigma$ may change the value of $\langle Q^\circ(\xi_{i_1}),\ldots, Q^\circ(\xi_{i_{3n-q}})\rangle_\Gamma$. However, one may see that its trace is invariant under a bordism of $\Sigma$, by exactly the same argument as \cite[Lemma~4.1]{Wa3}. We have
\[ Z_n(K_0)-Z_n(K_1)=\pm \frac{1}{2^n(2n)!(3n)!}\sum_\Gamma \partial I_\Gamma(\widetilde{K}). \]

It follows from the formula of $\partial Q(\wxi)$ in Theorem~\ref{thm:propagator} that we may arrange that the possible contributions in the boundary of $I_\Gamma(\widetilde{K})$ in the 1-parameter family are of the following forms:
\begin{enumerate}
\item Z-graph with an edge in $\Gamma$ collapsed to a point.
\item Z-graph with a {\it nonseparated full} subgraph $\Gamma'$ of $\Gamma$ with at least 3 vertices collapsed to a point, where we say that $\Gamma'$ is {\it full} if every edge $(i,j)\in E^{nW}(\Gamma)$ with $i,j\in V(\Gamma')$ belongs to $E^{nW}(\Gamma')$ and if the vertices in $V^W(\Gamma')$ are successive in $\Gamma$, and we say that $\Gamma'$ is {\it nonseparated} if $\Gamma$ is not of the form $\mbox{Closure}(\Gamma')\# \Gamma''$ for any Jacobi diagram $\Gamma''$, which may be empty.
\item (Anomalous face) Z-graph with a separated component in $P(\Gamma)$ collapsed to a point on a knot, where we say that a component $\Gamma_0$ in $P(\Gamma)$ is separated if $\Gamma$ is of the form $\Gamma'\# \Gamma''$ or $\Gamma'$ for Jacobi diagram(s) $\Gamma'$ and $\Gamma''$ such that $\Gamma'$ is nullhomotopic\footnote{If one of two Jacobi diagrams is nullhomotopic, then the connected sum of the two is well-defined (i.e., independent of the choice of arcs for the connected sum).} and $P(\Gamma')=\{\Gamma_0\}$.
\end{enumerate}
These correspond to the intersection of $I_\Gamma(\widetilde{K})$ with the boundary strata of $\bConf_{t,q}(M\times I,\widetilde{K})$ that are not over $\{0,1\}$. Note that it is not necessary to consider a Z-graph with a non self-loop edge forming a closed Z-path in $M$ since such a Z-graph and a Z-graph with one 4-valent vertex do not occur simultaneously in a generic 1-parameter family. 

\begin{Lem}\label{lem:inv_1_2}
For $n\geq 1$, $Z_n$ is invariant under bifurcations (1) and (2).
\end{Lem}
\begin{proof}
The invariance under a bifurcation of type (1) follows by the IHX and STU relations. Roughly, the Z-graph with an edge labeled $k$ collapsed to a point has one 4-valent vertex whose contribution in $F_\Gamma(\Sigma;\xi_1,\ldots,\xi_{3n})$ is a rational function with no terms of nonzero exponents of $t_k$. The sum of contributions of such Z-graphs come from terms in the IHX or the STU relations and they are set to be zero by the relations. For the detail, see \cite[Theorem~1]{AF}, where the boundary contribution is given by integrations in place of the counts. 

The invariance under a bifurcation of type (2) follows by an analogue of Kontsevich's lemma (\cite{Ko1}) as in \cite{Wa3} and by dimensional reasons. See the proof of \cite[Lemma~4.5]{Wa3} for the detail. Here, we must take care of the orientations of the faces of the moduli spaces at the boundary of $\bConf_\Gamma(M\times I,\widetilde{K})$. However, the problem is local and the orientations of the faces can be treated in almost the same way as the case of knots in $S^3$ given in \cite{BT, AF}. 
\end{proof}

Next, we shall consider the invariance under a bifurcation of type (3). Let $s_0\in I$ be a parameter at which a bifurcation of type (3) occurs and put $J=[s_0-\ve,s_0+\ve]$ for $\ve>0$ small. For a generic 1-parameter family $\widetilde{K}$, we may assume that there are finitely many such bifurcation time $s_0$ in $I$. Moreover, we may assume that only one bifurcation occurs in $J$. The change of the value of $Z_n$ for the bifurcation is given by sum of contributions of the anomalous face $\calS$ of $\partial\bConf_\Gamma(M\times J,\widetilde{K})$, which can be described as follows, following \cite{BT, BC}. 

Let $\Gamma$ be a labeled Jacobi diagram on $S^1$ with $q$ univalent and $t$ trivalent vertices. We equip $\Gamma$ with a $\Lambda$-coloring that is nullhomotopic. Suppose for simplicity that the univalent vertices of $\Gamma$ are labeled by $1,2,\ldots,q$ and they are cyclically ordered in this order. Let $P\to M$ be the orthonormal frame bundle associated to $TM$. Then the unit tangent bundle $S(TM)$ is identified with $P\times_{SO(3)}S^2$. Let $\bar\beta:S(\beta^*TM)\to S(TM)$ be the natural bundle map that covers $\beta=\widetilde{K}:S^1\times J\to M$. Let $B_{t,q}$ be the space of points $(a,u_1,\ldots,u_q,v_1,\ldots,v_t)$ in $S^2\times (\R)^q\times (\R^3)^t$ such that
\begin{itemize}
\item $u_1<\cdots<u_q$, or its cyclic permutations,
\item $v_i\neq v_j$ if $i\neq j$,
\item $u_ia\neq v_j$ for all $i\in\{1,\ldots,q\},j\in\{1,\ldots,t\}$,
\item $\displaystyle\sum_{i=1}^q u_i^2+\sum_{j=1}^t \|v_j\|^2=1$ and $\displaystyle\sum_{i=1}^q u_i+\sum_{j=1}^t \langle v_j,a\rangle =0$.
\end{itemize}
The forgetful map $B_{t,q}\to S^2$ is a fiber bundle. Let $\nu:S^1\times J\to S(\beta^*TM)$ be the section given by the unit tangent vectors of knots $\widetilde{K}|_{S^1\times\{s\}}:S^1\times\{s\}\to M$ for each $s\in J$ and let $\calS$ be the pullback in the following commutative diagram. 
\[ \xymatrix{
  \calS \ar[rr]^-{E(\bar\beta\circ{\nu})} \ar[d] && P\times_{SO(3)}B_{t,q} \ar[d]\\
  S^1\times J \ar[rr]_-{\bar\beta\circ{\nu}} && S(TM) 
}\]
Then $\calS$ can be naturally identified with the interior of the anomalous face of $\partial\bConf_{t,q}(M\times J,\widetilde{K})$. 

Let $\pi_{ij}:P\times_{SO(3)}B_{t,q}\to S(TM)$ denote the fiberwise Gauss map, namely, it gives the direction of the straight line that connects the $i$-th and the $j$-th vertex in $T_xM$. Let
\[ \theta_\ell=E(\bar\beta\circ{\nu})^{-1}\pi_{ij}^{-1}(s_{\xi_\ell}^*(M)), \]
where the edge of $\Gamma$ labeled $\ell$ is $(i,j)$. This is a piecewise smooth submanifold of $\calS$ and has a natural coorientation induced from that of $s_{\xi_\ell}^*(M)$ in $S(TM)$. Let $\Gamma_0$ be a separated component in $P(\Gamma)$, let $k$ be the number of edges of $\Gamma_0$ in $E^{nW}(\Gamma)$ and let $\langle \theta_1,\ldots,\theta_k\rangle_{\Gamma_0}$ be the count of $\bigcap_{\ell=1}^k \theta_\ell$ with signs and put
\[ I_{\Gamma_0}(\widetilde{K})=\langle \theta_1,\ldots,\theta_k\rangle_{\Gamma_0}[\underline{\Gamma_0}] \in \calA_{n_0}(S^1;\Q), \]
where $\underline{\Gamma_0}$ is the union of $\Gamma_0$ and the Wilson loop, and $n_0=\mathrm{deg}\,\underline{\Gamma_0}$. If $\ve$ is small enough, then the interior boundary of $I_\Gamma(\widetilde{K})$ for the bifurcation of type (3) over $\mathrm{Int}\,J$ is of the form $I_{\Gamma_0}(\widetilde{K})\# I_{\Gamma_1}(K_{s_0-\ve})$ for $\Gamma_1$ such that $\Gamma=\underline{\Gamma_0}\#\Gamma_1$. 

\begin{Lem}\label{lem:anomaly}
For $n=1,2$, $Z_n$ is invariant under bifurcation (3).
\end{Lem}
\begin{proof}
  By the same reason as \cite[Theorem~1.6]{BT}, the sum of the interior boundary of $I_\Gamma(\widetilde{K})$ for all labeled Jacobi diagram $\Gamma$ and for all possible edge orientations vanishes if the degree of $\underline{\Gamma_0}$ is even. Note that in \cite{BT}, a cancellation by symmetry of a (unlabeled) graph is considered, whereas we consider a cancellation between two graphs with different labellings. For example, let $\Gamma_0$ (resp. $\Gamma_0'$) be a labeled Jacobi diagram as in the left side (resp. right side) of Figure~\ref{fig:Y-symmetry}. Note that the graph-orientation for $\Gamma_0$ and $\Gamma_0'$ are equivalent because $\Gamma_0'$ is obtained from $\Gamma_0$ by a swap of vertices labeled 2 and 4 and by reversing all the three edge-orientations. We denote by $B_{1,3}(a)$ the fiber of the bundle $B_{1,3}\to S^2$ over $a\in S^2$. Let $x=(u_2,u_3,u_4,v_1)\in B_{1,3}(a)$ be a configuration for the graph $\Gamma_0$ and let $V\in\bigwedge^4 T^*_xB_{1,3}(a)$ be a local volume form on the fiber $B_{1,3}(a)$ which gives the orientation of the fiber. Suppose for simplicity that $x$ is represented by coordinate with $u_2=0$. We consider the automorphism $s:B_{1,3}(a)\to B_{1,3}(a)$ defined by $s(u_i)=-u_i$ ($i=3,4$), $s(v_1)=-v_1$. Then $s$ maps $x$ to a configuration $y=(0,u_3',u_4',v_1')=(0,-u_3,-u_4,-v_1)\in B_{1,3}(a)$ for $\Gamma_0'$. Let $V_y\in\bigwedge^4 T^*_yB_{1,3}(a)$ be a local volume form that is compatible with $V_x$. Then $s$ reverses the orientation of $B_{1,3}(a)$. Let $\theta_\ell^*$ denote the coorientation of $\pi_{ij}^{-1}(s_\xi^*(M))$ in $B_{1,3}$. We have
\[ \begin{split}
  &\langle (\theta_1^*)_{(v_1,0)}\wedge (\theta_2^*)_{(v_1,au_3)}\wedge (\theta_3^*)_{(v_1,au_4)}, V_x\rangle\\
  &=\langle s^*((\theta_1^*)_{(0,v_1')}\wedge (\theta_2^*)_{(au_3',v_1')}\wedge (\theta_3^*)_{(au_4',v_1')}), V_x\rangle\\
  &=\langle (\theta_1^*)_{(0,v_1')}\wedge (\theta_2^*)_{(au_3',v_1')}\wedge (\theta_3^*)_{(au_4',v_1')}, ds_*V_x\rangle\\
  &=-\langle (\theta_1^*)_{(0,v_1')}\wedge (\theta_2^*)_{(au_3',v_1')}\wedge (\theta_3^*)_{(au_4',v_1')}, V_y\rangle.
\end{split} \]
This implies that the sum of terms for all possible labellings and edge orientations vanishes. 

\begin{figure}
\fig{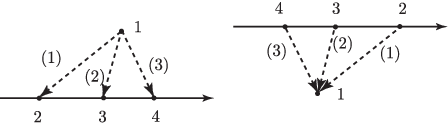}
\caption{}\label{fig:Y-symmetry}
\end{figure}

If the degree of $\underline{\Gamma_0}$ is 1, the contributions of the graphs that are relevant to the anomalous faces vanish by the FI relation. This completes the proof.
\end{proof}

\subsection{Bifurcations of Z-graphs in 1-parameter family, $\xi_i$ perturbed}

\begin{Lem}\label{lem:inv_concordance}
For $n=1,2$, $Z_n$ is invariant under a concordance of $\xi_i$.
\end{Lem}
\begin{proof}
Let $\widetilde{\xi}_i$ be the 1-parameter family for a concordance of $\xi_i$. By \cite[\S{4.3}--\S{4.6}]{Wa3}, the equivariant propagators of Z-paths extend over the 1-parameter family. Using the extensions, we may define the piecewise smooth 1-chain
\[ I_\Gamma(\widetilde{K})=\Tr_\Gamma\langle \widetilde{Q}^\circ(\widetilde{\xi}_{i_1}),\ldots, \widetilde{Q}^\circ(\widetilde{\xi}_{i_{3n-q}})\rangle_\Gamma \]
in $C_1(\bConf_{t,q}(M\times J,\widetilde{K});\Q)\otimes_\Q\calA_n(S^1;\widehat{\Lambda})$. For the rest of the proof, one may see by an argument similar to \cite{Wa3} that the boundaries of $I_\Gamma(\widetilde{K})$ that may contribute are of the same kinds as (1)--(3) listed above. Proof of the invariance under bifurcations of types (1) and (2) is the same as Lemma~\ref{lem:inv_1_2}. 

To prove (3), assume that $n\leq 2$. The anomalous face $\widetilde{\calS}$ over a closed interval $J\subset I$ is the pullback in the following commutative diagram.
\[ \xymatrix{
  \widetilde{\calS} \ar[rr]^-{E(\widetilde{\nu})} \ar[d] && (P\times_{SO(3)}B_{t,q})\times J \ar[d]\\
  S^1\times J \ar[rr]_-{\widetilde{\nu}} && S(TM)\times J
}\]
where $\widetilde{\nu}:S^1\times J\to S(TM)\times J$ is the section on $\widetilde{K}$ given by the unit tangent vectors of knots $\widetilde{K}|_{S^1\times\{s\}}:S^1\times\{s\}\to M\times\{s\}$. Let $(M\times J)_0=(M\times J)\setminus\bigcup_{\sigma\,:\,\mathrm{critical\, locus\, of\, \widetilde{\xi}_\ell}}\sigma$ and let $s_{\widetilde{\xi}_\ell}:(M\times J)_0\to S(TM)\times J$ be the normalization $-\widetilde{\xi}_\ell/\|\widetilde{\xi}_\ell\|$ of the section $-\widetilde{\xi}_\ell$. The closure $\overline{s_{{\widetilde{\xi}_\ell}}((M\times J)_0)}$ in $S(TM)\times J$ is a smooth manifold with boundary whose boundary in $S(TM)\times \mathrm{Int}\,J$ is the disjoint union of circle bundles over the critical loci $\sigma$ of $\widetilde{\xi}_\ell$ (including the birth-death locus $\gamma$). The fibers of the circle bundles are equators of the fibers of $S(T\sigma)$. Let $E^-_\sigma$ be the total space of the 2-disk bundle over $\sigma$ whose fibers are the lower hemispheres of the fibers of $S(T\sigma)$ which lie below the tangent spaces of the level surfaces of $\kappa$. Then $\partial\overline{s_{\widetilde{\xi}_\ell}((M\times J)_0)}=\bigcup_\sigma \partial E_\sigma^-$ as sets. Let
\[ s_{\widetilde{\xi}_\ell}^*(M\times J)=\overline{s_{\widetilde{\xi}_\ell}((M\times J)_0)}\cup \bigcup_\sigma E^-_\sigma\subset S(TM)\times J.\]
This is a 4-dimensional piecewise smooth manifold. We orient $s_{\widetilde{\xi}_\ell}^*(M\times J)$ by extending the natural orientation $(s_{\widetilde{\xi}_\ell}^{-1})^*o(M\times J)$ on $s_{\widetilde{\xi}_\ell}((M\times J)_0)$. Let 
\[ \widetilde{\theta}_\ell =E(\widetilde{\nu})^{-1}\widetilde{\pi}_{ij}^{-1}(s_{\widetilde{\xi}_\ell}^*(M\times J)), \]
where $\widetilde{\pi}_{ij}:(P\times_{SO(3)}B_{t,q})\times J\to S(TM)\times J$ denote the fiberwise Gauss map. We put
\[ \widetilde{I}_{\Gamma_0}(\widetilde{K})=\langle \widetilde{\theta}_1,\ldots,\widetilde{\theta}_k\rangle_{\Gamma_0}[\underline{\Gamma_0}]\in \calA_{n_0}(S^1;\Q), \]
where $\underline{\Gamma_0}$ is the union of $\Gamma_0$ and the Wilson loop, and $n_0=\deg\,\underline{\Gamma_0}$. If a bifurcation of type (3) occurs at $s_0$ and $J=[s_0-\ve,s_0+\ve]$ small enough, then the interior boundary of $I_\Gamma(\widetilde{K})$ for the bifurcation of type (3) is of the form $\widetilde{I}_{\Gamma_0}(\widetilde{K})\# I_{\Gamma_1}(K_{s_0-\ve})$ for $\Gamma_1$ such that $\Gamma=\underline{\Gamma_0}\#\Gamma_1$. The reason of the vanishing of the sum of the interior boundary of $I_\Gamma(\widetilde{K})$ is the same as Lemma~\ref{lem:anomaly}.
\end{proof}

Theorem~\ref{thm:invariance} now follows as a corollary of Lemmas~\ref{lem:inv_1_2}, \ref{lem:anomaly} and \ref{lem:inv_concordance}.

\mysection{Proof of Theorem~\ref{thm:injective}}{s:proof}

\subsection{Surgery formula}

\begin{Thm}\label{thm:surgery-formula}
Let $\kappa:M\to S^1$ be a smooth bundle with fiber diffeomorphic to an oriented connected closed surface. For $n\leq 2$, $Z_n$ satisfies the following properties.
\begin{enumerate}
\item $Z_n$ is a finite type invariant of $M$-null type $n$. Hence $Z_n$ induces a map $Z_n:\calK_n(M)/\calK_{n+1}(M)\to \calA_n(S^1;\widehat{\Lambda})$.
\item If $H_1(M)=\Z$, then $Z_n(\psi_n(\Gamma(\phi)))=[\Gamma(\phi)]$ in $\calA_n(S^1;\widehat{\Lambda})$ for a monomial $\Lambda$-colored Jacobi diagram $\Gamma(\phi)\in\calG_n^\null(S^1;\Lambda)$. For general $H_1(M)$, the following diagram is commutative:
\[ \xymatrix{
  \calA_n^\null(S^1;\Lambda_M) \ar[r]^-{\psi_n} \ar[d]_-{\kappa_*} & \calK_n(M)/\calK_{n+1}(M) \ar[ld]^-{Z_n}\\
  \calA_n(S^1;\widehat{\Lambda})}
\]
where $\kappa_*$ is the natural map induced by $\kappa_*:K(H_1(M),1)\to K(H_1(S^1),1)$.
\end{enumerate}
\end{Thm}
\begin{proof}
We shall prove the theorem only for $n=2$. The case $n=1$ is similar and easier. The following proof is analogous to the case of knots in $S^3$ given in \cite{BT, AF}.

(1) It suffices to check $Z_2(\calK_{3,3}(M))=0$. Let $[K;G]\in\calK_{3,3}(M)$, $G=\{G_1,G_2,G_3\}$, be an $M$-null forest scheme consisting of three disjoint $I$-claspers and let $R_i$ be a regular neighborhood of $G_i$ in $M$. (If $G_i$ is a strict $I$-clasper, then $R_i$ is an open ball. If $G_i$ has one nullhomologous leaf, then $R_i$ is an open solid torus.) By shrinking strict $I$-claspers by isotopy, we may assume that $R_i$'s are mutually disjoint and that $R_i$ is a small ball in $M$ if $G_i$ is a strict $I$-clasper. We show that $I_\Gamma([K;G])=\sum_{I\subset\{1,2,3\}}(-1)^{3-|I|}I_\Gamma(K^{G_I})$ vanishes for such a forest scheme $[K;G]$ and for any Jacobi diagram $\Gamma$ of degree $2$. 

Recall from Proposition~\ref{prop:F} that $I_\Gamma$ is given by counts of Z-graphs. By Lemma~\ref{lem:prop-scheme} (2), Z-graphs $X$ of degree 2 that may contribute to the alternating sum $I_\Gamma([K;G])$ should be such that for every $i$, $R_i$ is {\it occupied}, i.e., on each component of $R_i\cap K^{G_I}$ there is at least one univalent vertex of $X$. Note that if $G_i$ is a strict $I$-clasper, then $R_i\cap K^{G_i}$ consists of two components, and if $G_i$ is an $I$-clasper with a nullhomologous leaf, then $R_i\cap K^{G_i}$ consists of one strand. Hence $X$ should have at least 5 univalent vertices. But this is impossible if $X$ is of degree 2. This completes the proof of $Z_2([K;G])=0$ and of (1).

(2) It suffices to check the assertion for strict forest schemes in $\calK_{2,2}(M)$ since any $\Lambda$-colored Jacobi diagram on $S^1$ can be written as a sum of chord diagrams by the STU relation and since $\psi_n$ and $Z_n$ are linear. Let $[K;G]\in\calK_{2,2}(M)$, $G=\{G_1,G_2\}$, be a forest scheme corresponding to a chord diagram $\Gamma(\phi)$ such that $G_1$ and $G_2$ are strict $I$-claspers. Moreover, we may assume that the $I$-claspers are shrunk into small balls. By the argument of (1), the Z-graphs $X$ that may survive in the alternating sum $Z_2([K;G])$ should be such that for $i=1,2$, $R_i$ is occupied. Such a Z-graph corresponds to a chord diagram. Furthermore, the two univalent vertices in each $R_i$ should be connected by a Z-path included in $R_i$, i.e., a Z-path without horizontal segments, because if not, the small crossing change in $R_i$ does not change the value of $I_\Gamma$. 

For the monomial $\Lambda$-colored chord diagram $\Gamma(\phi)$, there are $\frac{2^2 4!\,6!}{|\Aut{\Gamma(\phi)}|}$ different ways of labelings and edge-orientations on $\Gamma(\phi)$, where $|\Aut{\Gamma(\phi)}|$ is the order of the group of automorphisms $g$ of $\Gamma(\phi)$ which preserves the homotopy class of $\Lambda$-coloring, namely, for two colorings $c_1,c_2:\Gamma\to K(H_1(M),1)$, $g:\Gamma(c_1)\to \Gamma(c_2)$ is such that $[c_1]=[c_2\circ g]$. Each term of these labeled oriented chord diagrams in the alternating sum $Z_2([K;G])$ contribute as $|\Aut\Gamma(\phi)|\,[\Gamma(\phi)]$. The terms for other graphs vanish. Hence we have
\[ Z_2([K;G])=\frac{1}{2^2 4!\,6!}\frac{2^2\,4!\,6!}{|\Aut\Gamma(\phi)|}|\Aut\Gamma(\phi)|\,[\Gamma(\phi)]=[\Gamma(\phi)]. \]
This completes the proof of (2).
\end{proof}

\begin{proof}[Proof of Theorem~\ref{thm:injective}]
Theorem~\ref{thm:injective}(1) is now immediate from Theorem~\ref{thm:surgery-formula}(2) and Lemma~\ref{lem:natural-injective}.
\end{proof}

\begin{Exa}
Let $M=S^2\times S^1$ and $K=\{p_0\}\times S^1$ ($p_0\in S^2$). Let $O$ be the unknot in a small ball in $M$. We consider the Whitehead double $\mathrm{Wh}\,(K)$ with respect to the product framing on $K$. By Theorem~\ref{thm:surgery-formula}, we see that 
\[ Z_1(\mathrm{Wh}\,(K))-Z_1(O)=[\Theta(0,1)], \]
where $\Theta(p,q)$ is the $\Lambda$-colored Jacobi diagram defined in Appendix~\ref{s:A_1} below. It is easy to check that $Z_1(O)=0$. For example, put $O$ in a level surface locus of the fiberwise Morse function $f$ for $\kappa$ so that it is disjoint from all the loci of descending/ascending manifolds of index 1. Then there are no Z-paths between two distinct points on $O$. Hence we have
\[ Z_1(\mathrm{Wh}\,(K))=[\Theta(0,1)]. \]
It follows from the result of Appendix~\ref{s:A_1} that $Z_1(\mathrm{Wh}\,(K))$ is nontrivial.
\end{Exa}

\appendix

\mysection{The structure of $\calA_1^\null(S^1;\Lambda)$}{s:A_1}

For $p,q\in \Z$, we put
\[ \Theta(p,q)=\fig{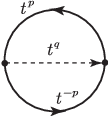},\quad \Omega(p)=\fig{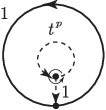}. \]

\begin{Lem}\label{lem:Omega=0}
\begin{enumerate}
\item $[\Omega(p)]=0$ in $\calA_1^\null(S^1;\Lambda)$.
\item $\calA_1^\null(S^1;\Lambda)$ is spanned by $\{[\Theta(p,q)]\}_{p,q\in\Z}$.
\end{enumerate}
\end{Lem}
\begin{proof}
The claim (1) follows immediately by the STU relation. Since any Jacobi diagram in $\calA_1^\null(S^1;\Lambda)$ of degree 1 with a self-loop is equivalent to $\Omega(p)$ for some $p$ modulo the Holonomy relation and since $[\Omega(p)]=0$ by (1), $\calA_1^\null(S^1;\Lambda)$ is generated by chord diagrams of the form $\Theta(p,q)$. This completes the proof of (2).
\end{proof}

Consider $\Q$ as the linear subspace of $\Q[t]$ of constants. Let 
\[ W:\calA_1^\null(S^1;\Lambda)\to \Q[t]/\Q\]
be the linear map defined by
\[ W([\Theta(p,q)])=t^{p+q}\quad (\mbox{mod $\Q$}).\]
This is well-defined since $W$ respects all the relations for $\calA_1^\null(S^1;\Lambda)$ (except the STU relation, which imposes no restriction for this degree, by Lemma~\ref{lem:Omega=0}(1)).
\begin{Prop}
The map $W$ is a linear isomorphism.
\end{Prop}
\begin{proof}
Let $L:\Q[t]/\Q\to \calA_1^\null(S^1;\Lambda)$ be the linear map defined by $L(t^p)=[\Theta(0,p)]$ for $p\geq 0$, which is well-defined. We have $L(W([\Theta(p,q)]))=L(t^{p+q})=[\Theta(0,p+q)]=[\Theta(p,q)]$ by the Holonomy relation and we have $W(L(t^{p}))=W([\Theta(0,p)])=t^p$ (mod $\Q$). This completes the proof.
\end{proof}

\begin{Lem}\label{lem:natural-injective}
The natural map $\calA_1^\null(S^1;\Lambda)\to \calA_1(S^1;\widehat{\Lambda})$ is injective.
\end{Lem}
\begin{proof}
Let $\calA_1^\null(S^1;\widehat{\Lambda})$ be the subspace of $\calA_1(S^1;\widehat{\Lambda})$ spanned by chord diagrams with one chord and with Wilson edges colored by 1. By Lemma~\ref{lem:Omega=0} (1) and its analogue for $\calA_1(S^1;\widehat{\Lambda})$, the STU relation imposes no restriction for this degree. Thus one may see that $\calA_1^\null(S^1;\widehat{\Lambda})$ is linearly isomorphic to $\Q(t)/\Q$, by a similar argument as above.
\end{proof}


\section*{\bf Acknowledgments.}
The author is supported by JSPS Grant-in-Aid for Young Scientists (B) 26800041. He is grateful to Professor Kazuo Habiro for pointing out mistakes in an earlier version of this paper and for helpful comments on clasper theory. He also would like to thank Professor Efstratia Kalfagianni for giving him information on finite type invariants of knots in 3-manifolds. 
\par
\end{document}